\documentclass[11pt,a4paper,reqno]{amsart}
\pdfoutput=1
\usepackage{a4wide}
\usepackage{amsfonts,amsthm,amsmath,amssymb}
\usepackage{amscd,color,caption}
\usepackage{psfrag,graphicx,subfigure,float,subfloat}
\usepackage{import,algorithmic,algorithm2e}

\newtheorem*{thmA}{Theorem A}

\theoremstyle{plain}

\newtheorem{theorem}{Theorem}[section]
\newtheorem{lemma}[theorem]{Lemma}

\newtheorem{rem}{Remark}

\newcommand{\al}{\left(\alpha\right)}
\newcommand{\be}{\left(\alpha_2\right)}

\newcommand{\ab}{\left(\alpha_1,\alpha_2\right)}

\newcommand{\w}{\color{white}}
\newsavebox{\savepar}

\begin{document}

\title{The secant map applied to a real polynomial with multiple roots}

\date{\today}

\author{Antonio Garijo}
\address{Departament d'Enginyeria Inform\`atica i Matem\`atiques,
Universitat Rovira i Virgili, 43007 Tarragona, Catalonia.}
\email{antonio.garijo@urv.cat}
\author{Xavier Jarque}
\address{Departament de Matem\`atica Aplicada i An\`alisi, Universitat de
Barcelona, 08007 Barcelona, Catalonia.}
\email{xavier.jarque@ub.edu}

\thanks{This work has been partially supported by MINECO-AEI grants
MTM-2017-86795-C3-2-P and MTM-2017-86795-C3-3-P,
the Maria de Maeztu Excellence Grant MDM-2014-0445 of the BGSMath and the AGAUR grant 2017 SGR 1374}\

\begin{abstract}
\vspace{0.5cm}
We investigate the plane dynamical system given by the secant map applied to a polynomial $p$ having at least one multiple root of multiplicity $d>1$. We prove that the local dynamics around the fixed points associated to the roots of $p$ depend on the parity of $d$. 

\vglue 0.2truecm


\vglue 0.2truecm


\end{abstract}

\maketitle

\section{Introduction and statement of the results}

The main goal of this paper is to investigate the dynamical system generated by the so called {\it secant map}, or {\it secant method} when considering it as a root finding algorithm, applied to the real monic polynomial of degree $k \geq 2$,  
\begin{equation*}
p(x)= a_k x^k + a_{k-1}x^{k-1} + \cdots + a_1 x + a_0 , \ a_k=1, \ a_j\in \mathbb R,\ j=0,\ldots k-1,
\end{equation*}
under the presence of real multiple roots. The secant map  writes as
\begin{equation}\label{eq:secant}
S(x,y)=\left(y,y-p(y) \frac{x-y}{p(x)-p(y)}\right).
\end{equation}
We refer to \cite{Tangent} for a detailed discussion of the dynamics generated by $S$ when all real roots of $p$ are simple. As in \cite{Tangent} we consider $S\colon \mathbb R^2 \to \mathbb R^2$ (with {\it poles}), but of course there is a natural extension of this problem by assuming $p$ as a complex monic polynomial and thus  $S:\mathbb C^2 \to \mathbb C^2$. See \cite{BedFri} for a discussion on this context.

Let $\alpha$ be a root of $p$, and consider the set 
\begin{equation}\label{eq:basin}
\mathcal A(\alpha) = \{ (x,y) \in \mathbb R^2 \, | \, S^n(x,y) \to (\alpha,\alpha),\ \ \hbox{as} \ \ n \to \infty\}.
\end{equation} 
Because $S$ is a root finding algorithm it is natural to investigate the structure and distribution of the sets $\mathcal A(\alpha)$ for all roots of $p$; we notice that $S(\alpha,\alpha)=\left(\alpha,\alpha\right)$. From the numerical point of view points in $\mathcal A(\alpha)$ define {\it good} initial conditions converging to $\alpha$.

In the present work we assume that at least one real root of $p$, 
$\alpha \in \mathbb R$, has multiplicity $d\geq 2$, i.e. $p^{(j)}(\alpha)=0$ for $0 \leq j \leq d-1$ and $p^{(d)}(\alpha)\neq 0$. This case is interesting itself but it is also relevant when studding the bifurcation phenomena of several simple roots colliding together.  

\begin{thmA}\label{theo:basins}
Let $p$ be a real, monic polynomial and let $\alpha$ be a real multiple root of $p$ of multiplicity $d\geq 2$. Let $S$ be the secant map defined in (\ref{eq:secant}). The following statements hold. 
\begin{itemize}
\item[(a)] If $d$ is an odd number then the point $(\alpha,\alpha)$ belongs to $\mathcal A(\alpha)$. Indeed there is an open neighbourhood $U$ of 
$(\alpha,\alpha)$ such that $U\subset \mathcal A(\alpha)$.
\item[(b)] If $d$ is an even number then $(\alpha,\alpha)$ belongs to the boundary of $\mathcal A(\alpha)$. In fact, it belongs to the common boundary of all the basins of attraction associated to simple real roots of $p$, i.e.,  
\[
(\alpha,\alpha) \in \bigcap_{ \tau\in \mathbb R, \, p(\tau)=0,\, p^\prime(\tau)\ne 0 } \partial \mathcal A(\tau).
\]
\end{itemize} 
\end{thmA}

Theorem A  has several implications when we use the secant method as a root finding algorithm applied to a polynomial $p$ with multiple roots. If the multiplicity of the root $\alpha$ of $p$ is odd, it inherits the local dynamics as it was a simple root, i.e., all initial seeds in a small neighbourhood converge to $\left(\alpha,\alpha\right)$ (see Theorem A(a)). However if $\alpha$ is a multiple root of even multiplicity the local dynamics is quite different. Although {\it most} of the initial seeds near $\left(\alpha,\alpha\right)$ converge to it, there are nearby initial conditions converging to all simple real roots of $p$ (see Theorem A(b)). It seems plausible, and numerical experiments support it, that in fact $(\alpha,\alpha)$ belongs to the boundary of all roots of $p$, not only the simple ones.  As we said before, Theorem A will be also useful for studding the bifurcation phenomena coming  from the collision of several roots.     

In Figure \ref{fig:dyn_plane} we illustrate Theorem A applied to $p_d(x)=(x+2)x(x-1)^d,\ d=2,3,4,5$. Colours red, blue and green, correspond to seeds converging to the roots $x=1$, $x=0$, $x=-2$, respectively. According to Theorem A  the dynamical plane of $S_p$ near the  corresponding fixed point $(1,1)$ change drastically for different values of $d$. We notice that in Figures \ref{fig:dyn_plane}(b) and \ref{fig:dyn_plane}(d) there are green points near $(1,1)$ although it is difficult to see. White colour corresponds to an unbounded critical cycle (for a discussion see \cite{BedFri,Tangent}.  

\begin{figure}[ht]
    \centering
    \subfigure[\scriptsize{ $p(x)=(x+2)x(x-1)^2$.}]{
     \includegraphics[width=0.3\textwidth]{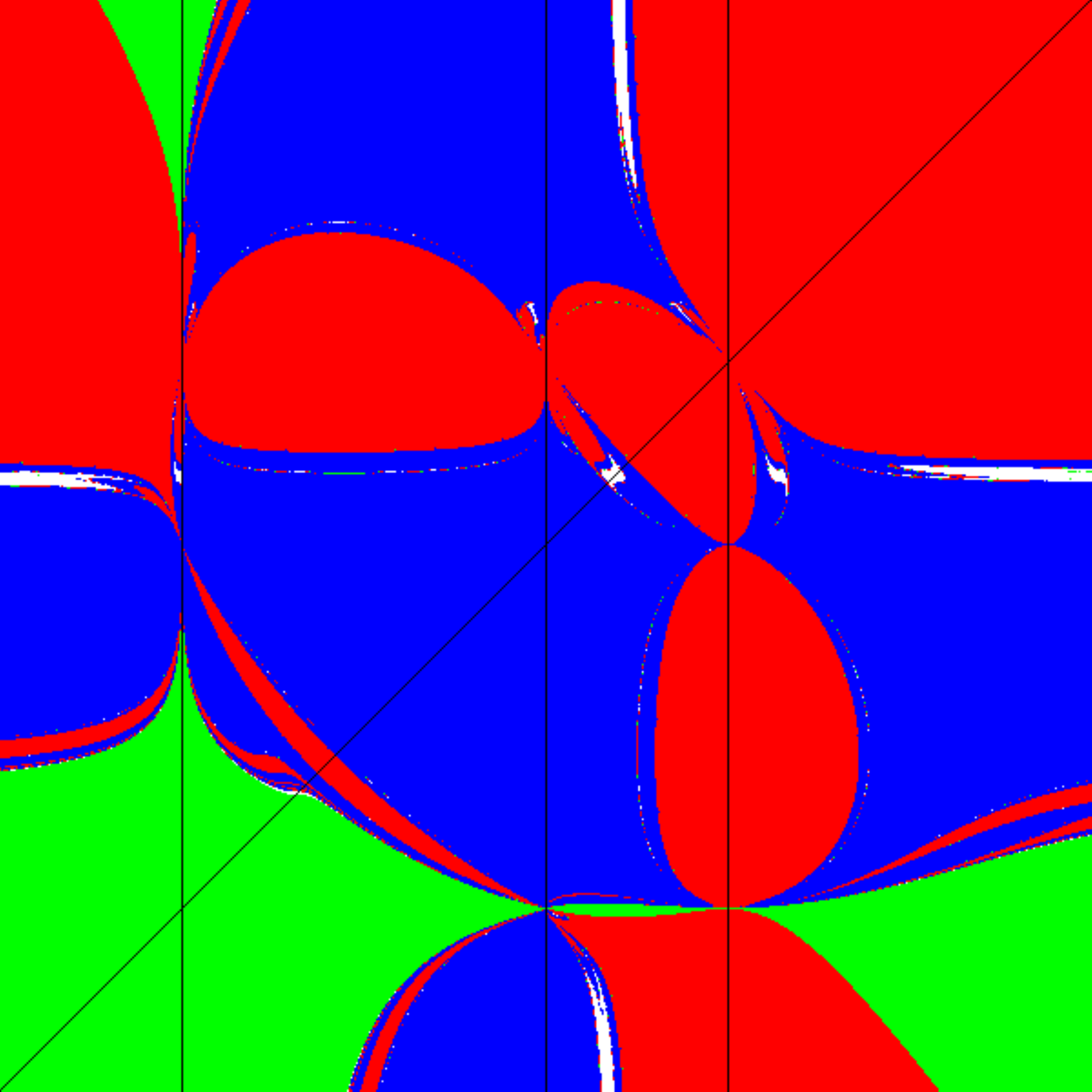}
         \put(-41,88) {{\tiny$(1,1)$}}
                  \put(-86,65) {\w{\tiny $(0,0)$}}
          \put(-110,20) {\tiny $(-2,-2)$}
      }
    \subfigure[\scriptsize{Zoom of (a) near $(1,1)$}]{
     \includegraphics[width=0.3\textwidth]{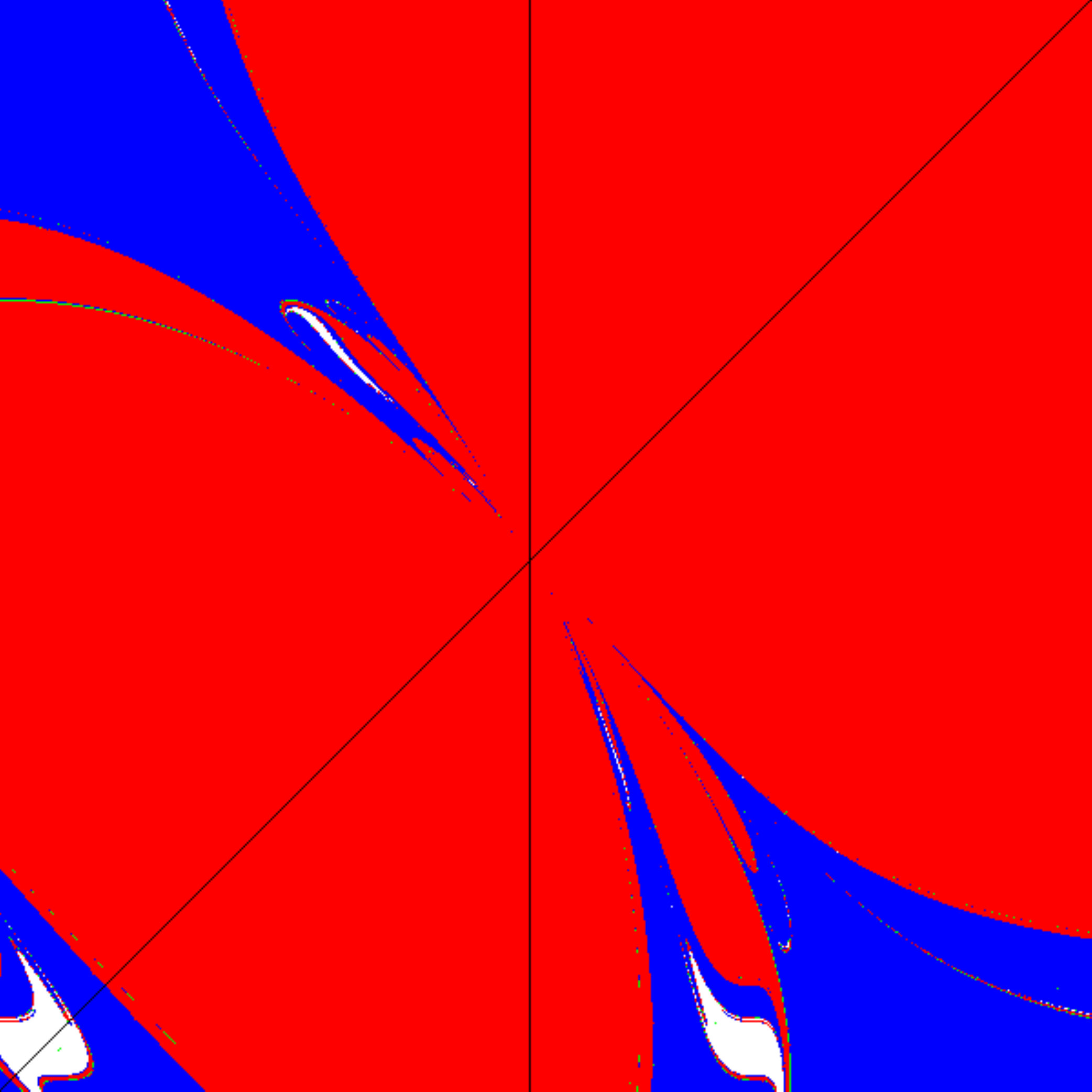}
         \put(-91,63) {{\tiny$(1,1)$}}
      }   \\
       \subfigure[\scriptsize{$p(x)=(x+2)x(x-1)^4$.}]{
     \includegraphics[width=0.3\textwidth]{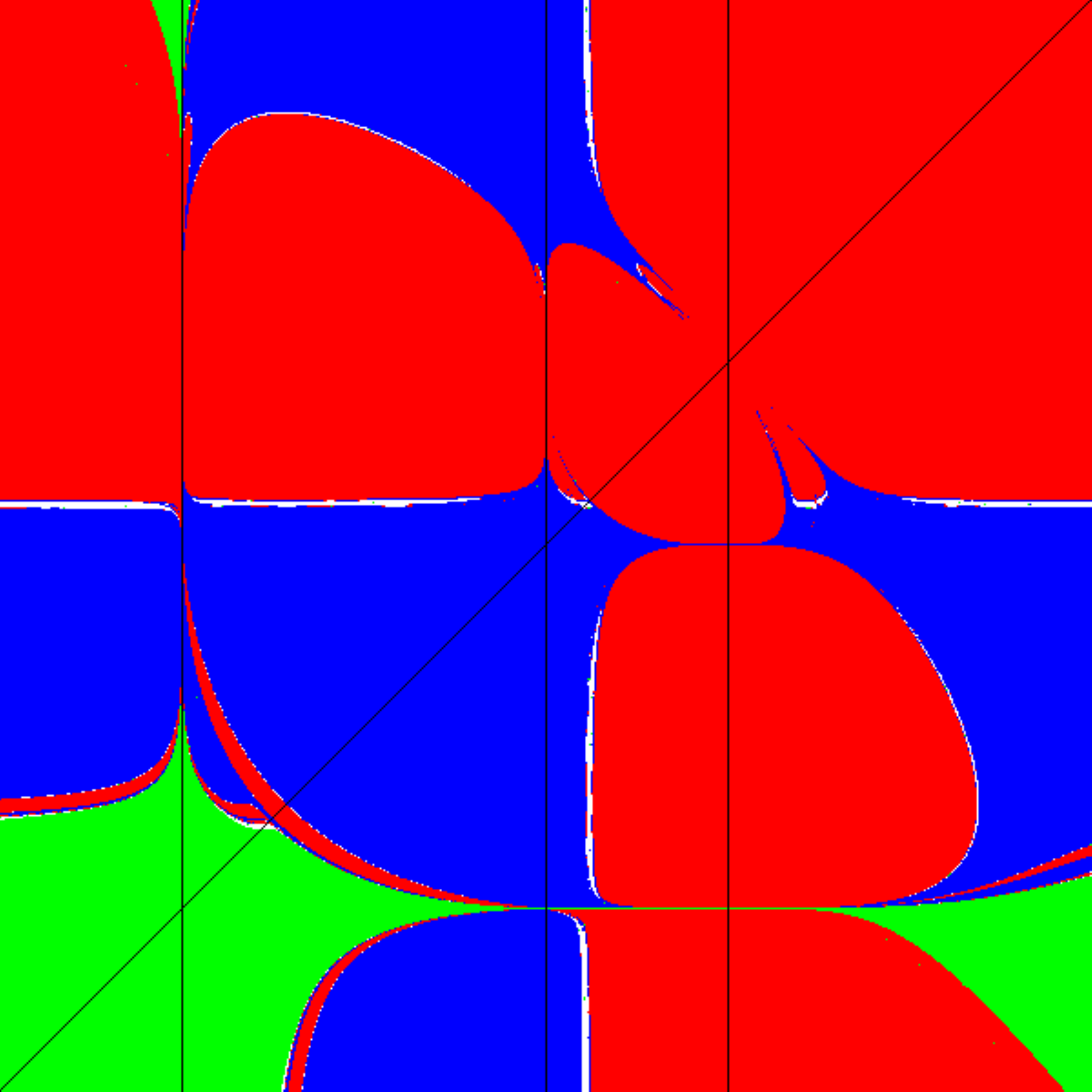}
    \put(-41,88) {{\tiny$(1,1)$}}
                  \put(-86,65) {\w{\tiny $(0,0)$}}
          \put(-110,20) {\tiny $(-2,-2)$}
      }
    \subfigure[\scriptsize{Zoom of (c) near $(1,1)$.}]{
     \includegraphics[width=0.3\textwidth]{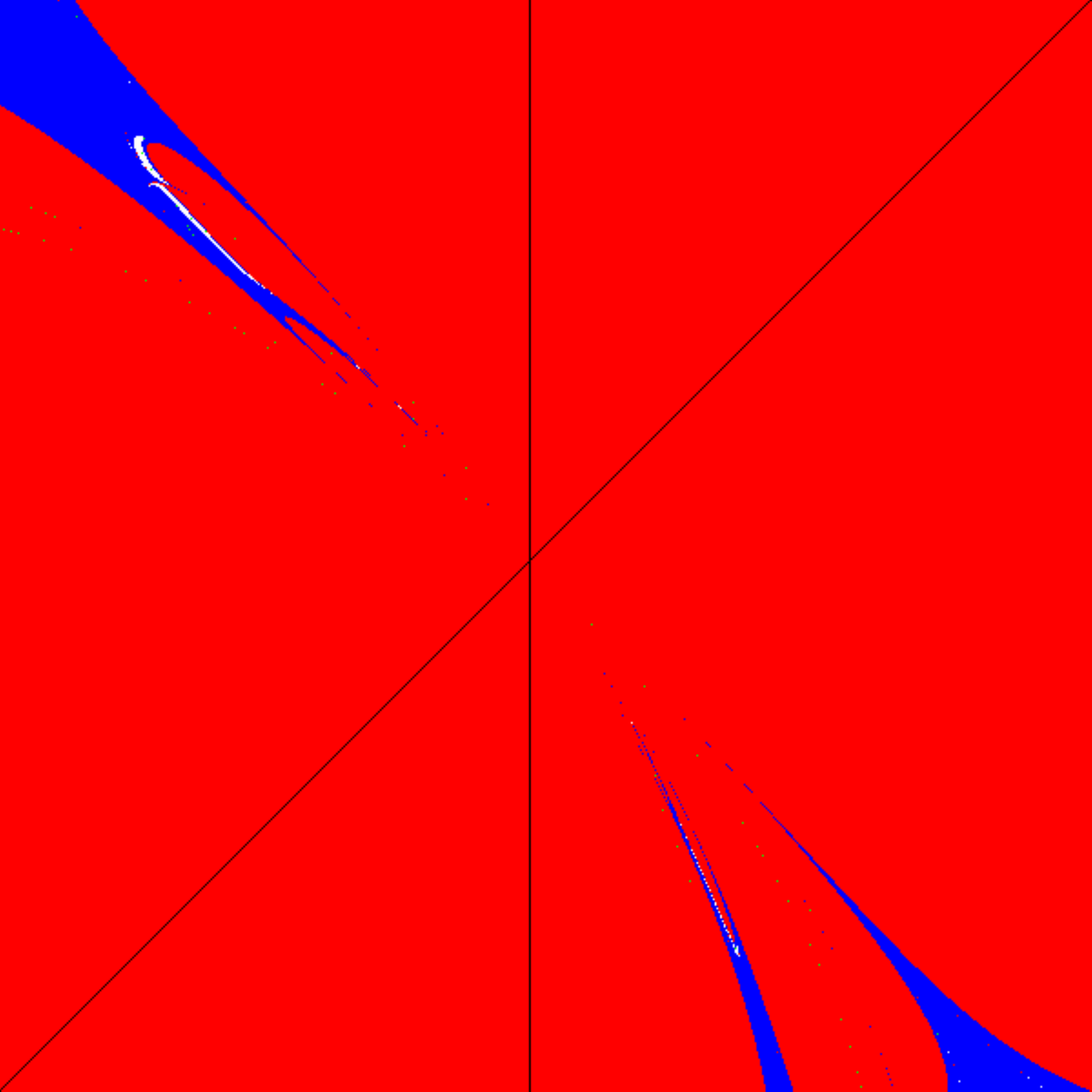}
         \put(-91,63) {{\tiny$(1,1)$}}
                       }
                       \\
       \subfigure[\scriptsize{$p(x)=(x+2)x(x-1)^3$.}]{
     \includegraphics[width=0.3\textwidth]{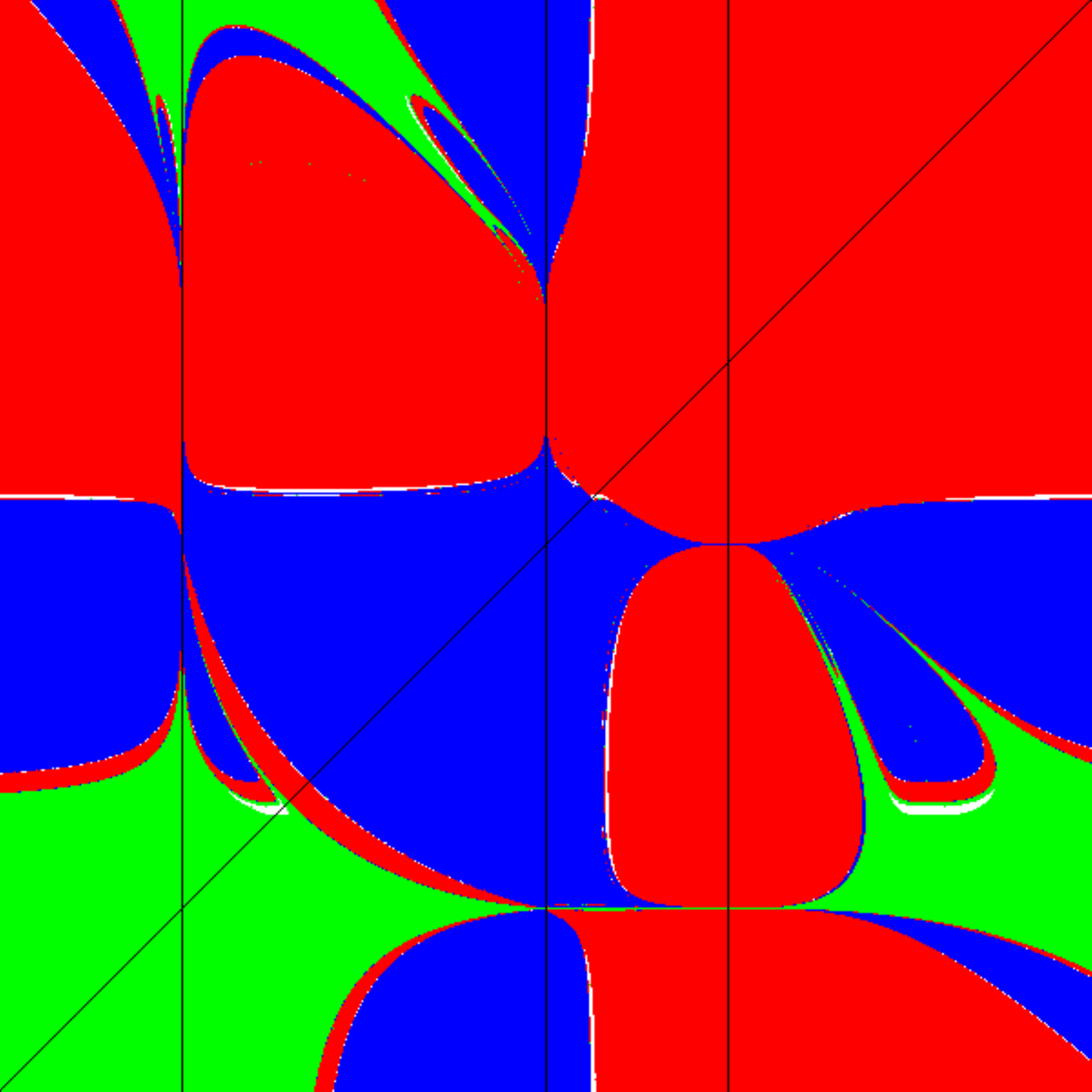}
    \put(-41,88) {{\tiny$(1,1)$}}
                  \put(-86,65) {\w{\tiny $(0,0)$}}
          \put(-110,20) {\tiny $(-2,-2)$}
      }
    \subfigure[\scriptsize{$p(x)=(x+2)x(x-1)^5$.}]{
     \includegraphics[width=0.3\textwidth]{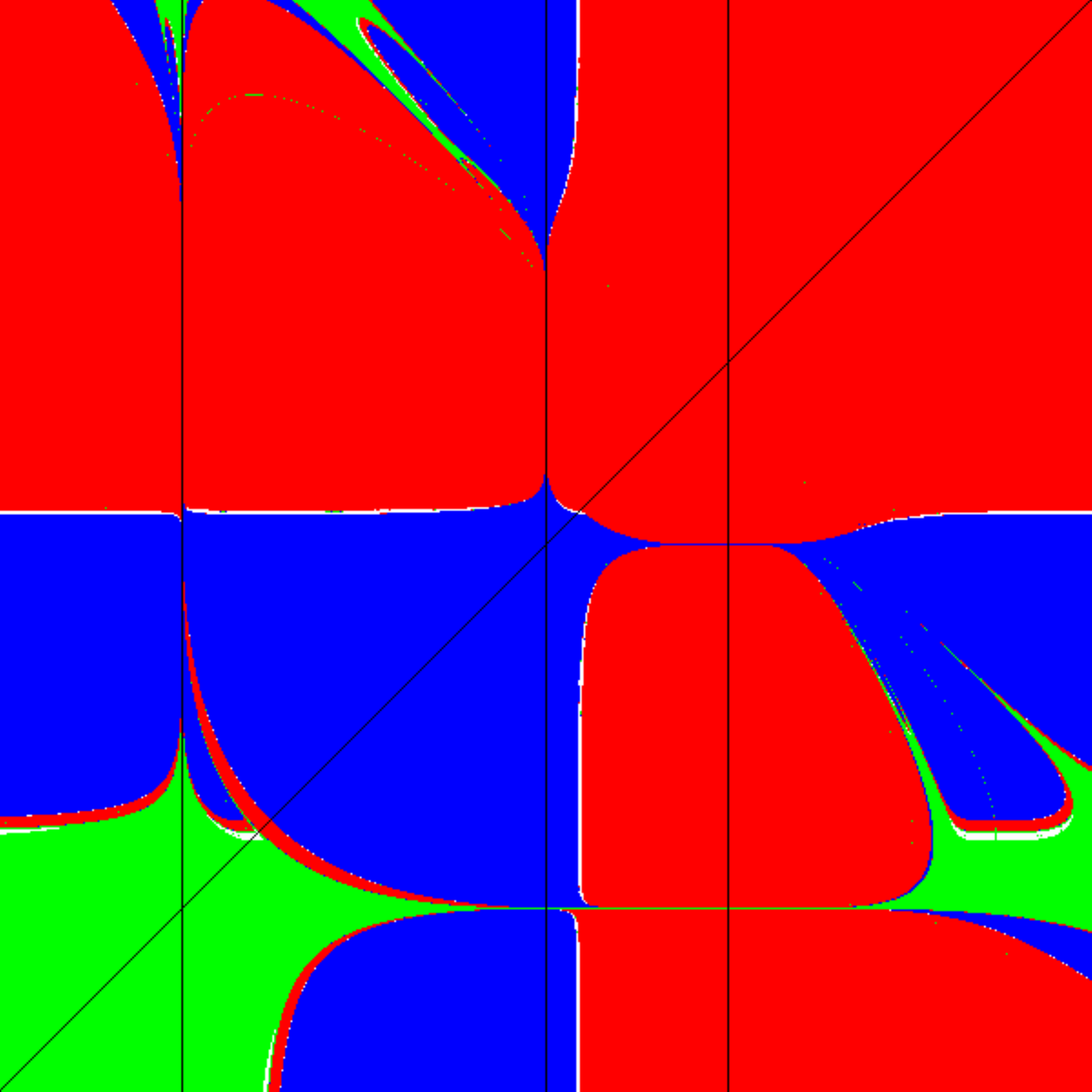}
    \put(-41,88) {{\tiny$(1,1)$}}
                  \put(-86,65) {\w{\tiny $(0,0)$}}
          \put(-110,20) {\tiny $(-2,-2)$}
          }
      
 \caption{\small{ Dynamical plane of the secant map applied to the family of polynomials $p(x)=(x+2)x(x-1)^d$ for several values of $d$.  We show in red (dark grey) the basin of attraction of the multiple root of $p$ corresponding to the fixed point of the secant map located at  $(1,1)$, in green (light grey) the basin of attraction of $(-2,-2)$ and in blue (black) the basin of attraction of $(0,0)$. The range of the pictures (a),(c),(e) and (f)  is [-3,3]x[-3,3]. }}
    \label{fig:dyn_plane}
    \end{figure}
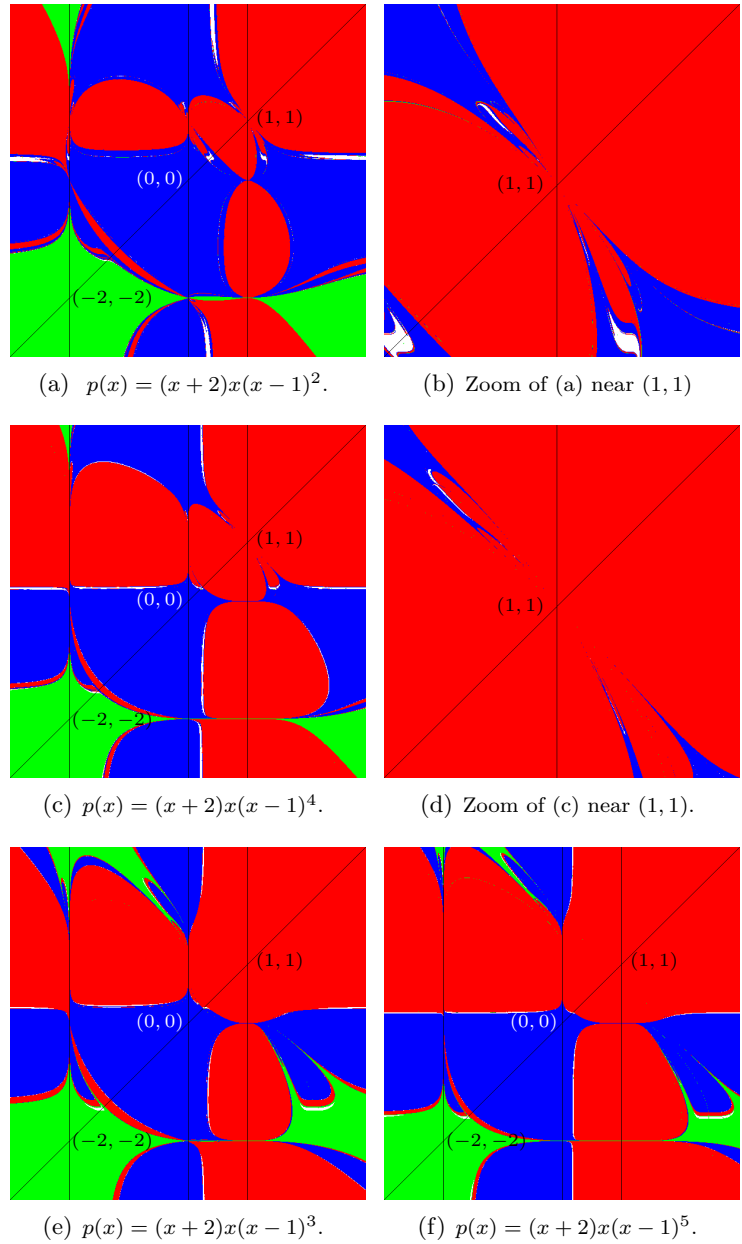

The paper is organized as follows. In Section 2 we introduce terminology and tools from a series of papers on rational iteration. In Sections 3 and 4 we compute the Taylor's polynomial associated to the secant map at some points, which is the main tool to prove the Theorem A. Finally Section 5 is devoted to prove Theorem A.

\section{Plane rational iteration}

For our purposes we follow the notation, and use some results and ideas, introduced and developed in the series of papers \cite{PlaneDenominator1,PlaneDenominator2,PlaneDenominator3}. Consider the plane rational map given by  
\begin{equation}
T: \left( 
\begin{array}{l}
x \\
y
\end{array}
\right) \mapsto  
 \left( 
\begin{array}{l}
F(x,y) \\
N(x,y)/D(x,y)
\end{array}
\right),
\label{eq:PlaneDenominator}
\end{equation}
where $F$, $N$ and $D$ are differentiable functions. Set  
\[
\delta_T = \{ (x,y) \in \mathbb R^2 \, | \, D(x,y)=0\} \quad {\rm and} \quad E_T =\mathbb R^2 \setminus \bigcup_{n \geq 0} T^{-n}(\delta_T).
\]
Easily $T=(T_1,T_2):E_T\to E_T$ defines a smooth dynamical system given by the iterates of $T$; that is $\{x_n:=T^n\left(x_0\right)\}_{n\geq 0},$ with $x_0\in T$. Clearly $T$ {\it sends} points of $\delta_T$ to infinity unless $N$ also vanishes. At those points where $T_2$ takes the form $0/0$, the definition of $T$ is uncertain in the sense that the value might depend on the path we choose to approach the point.  Although those uncertain points are outside ${E_T}$, they play a crucial role to understand the local and global dynamics of $T$.

We say that a point  $Q \in \delta_T \subset \mathbb R^2$ is a  {\bf focal point (of $T$)} if $T_2(Q)$ takes the form 0/0 (i.e. $N(Q)=D(Q)=0$), and there exists a smooth simple arc $\gamma:=\gamma(t),\ t\in (-\varepsilon,\varepsilon)$, with $\gamma(0)=Q$, such that $\lim_{t \to 0} T_2(\gamma)$ exists and it is finite. The line $L_Q=\{(x,y)\in \mathbb R^2 \ | \ x=F(Q)\}$ is called the {\bf prefocal line} (over $Q$). 

Let $\gamma$ passing through $Q$,  not tangent to $\delta_T$, with slope $m$ at $t=0$. Then $T\left(\gamma\right)$ will be a curve passing through some finite point $(F(Q),y(m)) \in L_Q$ at $t=0$ (see figure \ref{fig:focals_simple}). More precisely the value of $y(m)$ is given by 
\begin{equation}\label{eq:limit2}
y(m) = \lim_{t \to 0} \frac{N(\gamma(t))}{D(\gamma(t))}.
\end{equation}

A focal point $Q$ is defined by the intersection of two (algebraic) curves: $N(x,y)=0$ and $D(x,y)=0$. If they intersect transversally (at $Q$) we say that $Q$ is  a {\it simple focal point}; otherwise $Q$ is called a {\it non simple focal point}. In other words $Q$ is simple  if
$\nabla N(Q)=(N_x(Q),N_y(Q))$  and $\nabla D(Q)=(D_x(Q),D_y(Q))$ are linearly independent (i.e. $N_x(Q)D_y(Q)-N_y(Q)D_x(Q)\neq0$), while $Q$ is non-simple if 
$\nabla N(Q)$ and $\nabla D(Q)$ are linearly dependent, i.e.  
 $N_x(Q)D_y(Q)-N_y(Q)D_x(Q)=0$.

In the series of papers \cite{PlaneDenominator1,PlaneDenominator2,PlaneDenominator3}
the authors prove, among other things, many results to determine the sort of relationship between the slope $m$ of the curve $\gamma(t)$ at $t=0$ and the corresponding point $(F(Q),y(m))\in L_Q$ depending on the type of focal point. For instance if $Q$ is simple  (see \cite{PlaneDenominator1} for details) there is a one-to-one correspondence between the slope $m$ and points in the prefocal line $L_Q=\{(x,y)\in \mathbb R^2 \ | \ x=F(Q) \}$. We sketch the situation in Figure \ref{fig:focals_simple}.

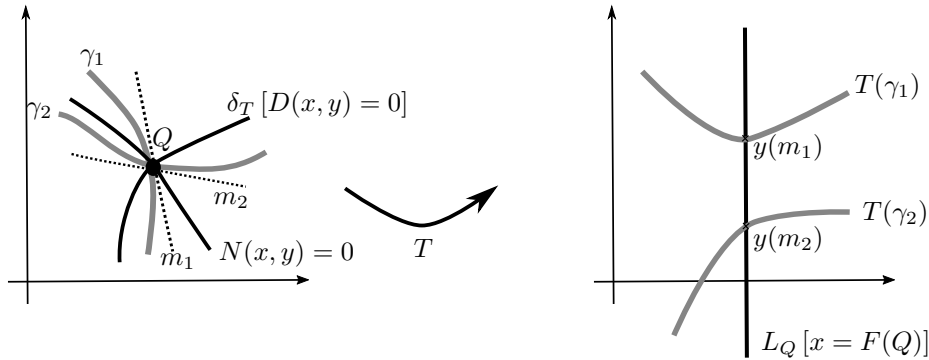
\begin{figure}[ht]
    \centering
     \includegraphics[width=0.75\textwidth]{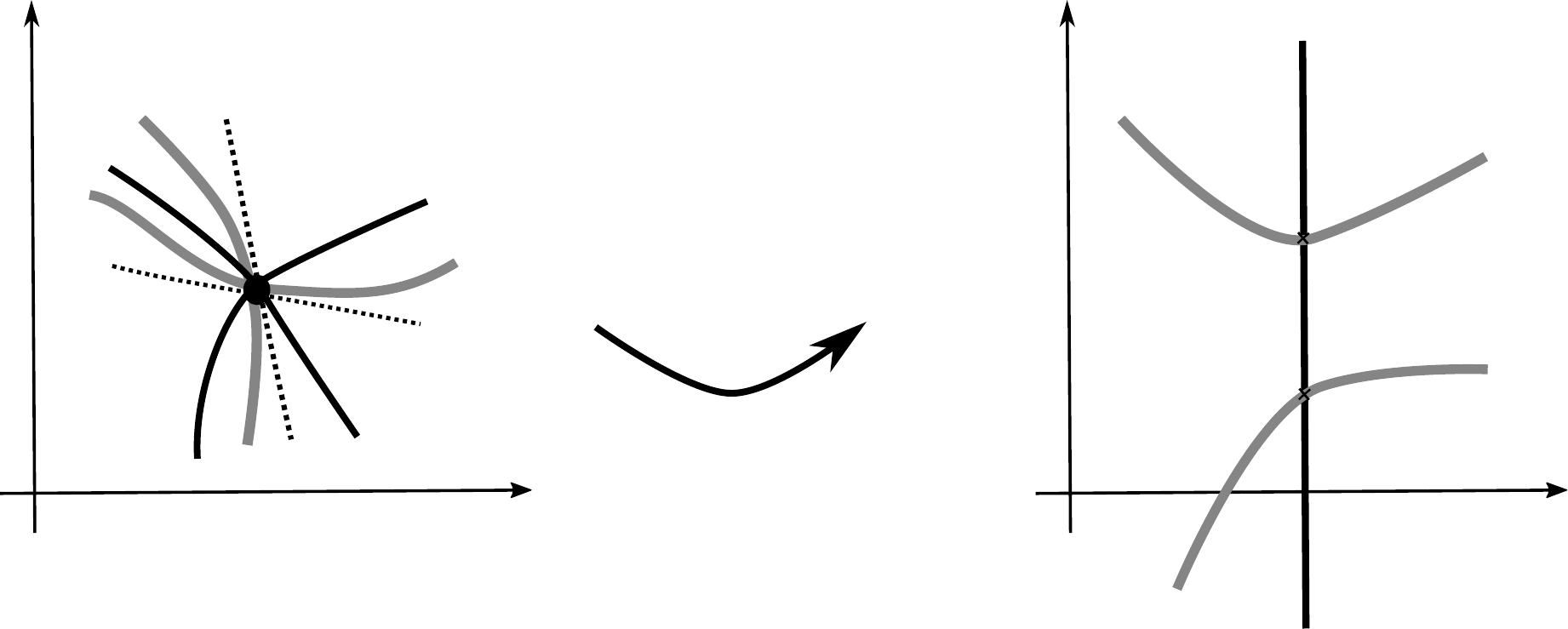}
    \put(-180,38) {\small $T$ } 
     \put(-305,112){\small $\gamma_1$}
     \put(-325,92){\small $\gamma_2$}
     \put(-255,58){\small $m_2$}
     \put(-275,35){\small $m_1$}
     \put(-278,80){\small $Q$}
     \put(-53,77){\small $y(m_1)$}
     \put(-53,42){\small $y(m_2)$}
     \put(-250,93){\small $ \delta_T\, [D(x,y)=0]$}
     \put(-253,37){\small $ N(x,y)=0$}
     \put(-15,100){\small $T(\gamma_1$)}
     \put(-12,52){\small $T(\gamma_2$)}
      \put(-50,3){\small $L_Q \, [  x = F(Q) ] $}
           \caption{\small{Dynamics of $T$ near a simple focal point $Q$.}}
    \label{fig:focals_simple}
    \end{figure}

If $Q$ is a non simple focal point the situation is more delicate (see \cite{PlaneDenominator3} for details). The authors studied the possible value(s) of the limit \eqref{eq:limit2}
depending on the precise algebraic conditions implying $N_x(Q)D_y(Q)-N_y(Q)D_x(Q)=0$. The major argument they used is to compute the Taylor's series of the functions $N(x,y)$ and $D(x,y)$ at the focal point $Q$. This is also our main tool here, adapted to the case of the secant map. Indeed when $\alpha$ is a multiple root of $p$ then the point $Q=(\alpha,\alpha)$ is a non simple focal point.

\begin{rem}
Focal points are also known as indeterminacy points in the general theory of several complex variables. 
\end{rem}

\section{Taylor's polynomials of the secant map}

 In this section we will present useful expressions of the secant map at the point $(\alpha,\beta)$ where both $\alpha$ and $\beta$ are roots of the polynomial $p$.  Set  $m\geq 1$ and 
\begin{equation} \label{eq:q-and-q_j}
\begin{split}
&q_m(x,y):=\sum_{\ell=0}^{m-1} x^{m-1-\ell}y^\ell, \, \ m=1,\ldots,k \\
& q(x,y):=\sum_{m=1}^k a_m q_m(x,y).
\end{split}
\end{equation}

\begin{lemma}[{\cite[Lemma 2.1 and 2.2]{Tangent}}]
\label{lem:pol_q}
The following statements hold.
\begin{itemize}
\item[(a)] For $m=1,\ldots k$ we have
$$
x^m-y^m=(x-y)q_m(x,y).
$$
\item[(b)] The (symmetric) polynomial $q(x,y)$ defined above satisfies 
$$
p(x)-p(y)=(x-y) q(x,y).
$$ 
In other words, the factor $(x-y)$ divides the expression $p(x)-p(y)$ and the resultant quotient is a (symmetric) polynomial of degree $k-1$. 
\item[(c)] The secant map defined in (\ref{eq:secant}) writes as 
\begin{equation}
\label{eq:def_S_ND}
S (x,y) = \left( y, \frac{y q(x,y)-p(y)}{q(x,y)}\right):= \left( y, \frac{N(x,y)}{D(x,y)}\right)
\end{equation}
for all  $(x,y)\in \mathbb R^2 \setminus \delta_S$. 
\end{itemize}
\end{lemma}

Next lemma gives precise Taylor's polynomials of $N(x,y)$ and $D(x,y)$ and hence of the rational map $S(x,y)$ at a point $(\alpha,\alpha)$, where $\alpha$ is a root of $p$ with multiplicity $d\geq 2$.

\begin{lemma}\label{lem:N_D}
Let $p$ be a polynomial of degree $k$ and let $\alpha$ be a root of $p$ of multiplicity $d$ with $2 \leq d \leq k-1$. Then,
$$
S (x,y) = \left( y, \frac{N(x,y)}{D(x,y)}\right)= \left( y, \alpha+\frac{N_1(x,y)}{D(x,y)}\right)
$$
where 
\begin{align}
D(x,y)= & \sum_{m=d}^{k}\frac{p^{(m)}\left(\alpha\right)}{m!} \sum_{\ell =0}^{m-1}  \, (x-\alpha)^{m-1-\ell} (y-\alpha)^{\ell},  \label{eq:Dalpha} \\
N_1(x,y)=&(x-\alpha)(y-\alpha) \sum_{m=d}^{k} \frac{1}{m!}\ p^{(m)}\left(\alpha\right) 
\sum_{\ell=1}^{m-1}(x-\alpha)^{m-1-\ell}(y-\alpha)^{\ell-1}
\label{eq:Nalpha}
\end{align} 
\end{lemma}

\begin{proof}
First we prove \eqref{eq:Dalpha}.  We claim that 
\[
D(x,y)=\sum_{m=1}^{k}\frac{p^{(m)}(x_0)}{m!} \sum_{\ell =0}^{m-1}  \, (x-x_0)^{m-1-\ell} (y-x_0)^{\ell}, \ x_0\in \mathbb R.
\]
Assuming that the claim is true, then \eqref{eq:Dalpha} follows immediately by substituting $x_0=\alpha$ where 
$\alpha$ satisfies $p^{(j)}(\alpha)=0$ for $0\leq j \leq d-1$ and $p^{(d)}(\alpha)\neq 0$. 

To see the claim observe that for any given $x_0\in \mathbb R$ we have
\[
p(x)  = \sum_{m=0}^k \frac{p^{(m)} (x_0)}{m!}  (x-x_0)^m \quad \hbox{ and }\quad p(y)  = \sum_{m=0}^k \frac{p^{(m)}(x_0) }{m!} (y-x_0)^m. 
\]
Then 
\[
D(x,y)=q(x,y)=\frac{p(y)-p(x)}{y-x}= \displaystyle \sum_{m=1}^k \frac{p^{(m)}(x_0) }{m!} \left[ \frac{  (y-x_0)^m - (x - x_0)^m } {(y-x_0)-(x-x_0)}\right].
\]
Using Lemma \ref{lem:pol_q}(a) we have that 
\[
D(x,y)=  \displaystyle \sum_{m=1}^k \frac{p^{(m)}(x_0)}{m!} q_m(x-x_0,y-x_0)=    \displaystyle \sum_{m=1}^k \frac{p^{(m)}(x_0)}{m!} \displaystyle \sum_{\ell =0}^{m-1} (x-x_0)^{m-1-\ell} (y-x_0)^{\ell} ,
\]
\noindent proving the claim. In particular we notice that
\begin{equation}\label{eq:derivD}
\left(\begin{array}{c}
m \\
\ell
\end{array}\right) \frac{\partial D^m}{\partial x^{m-\ell}y^{\ell}}\left(\alpha,\alpha\right)=\frac{1}{m+1}p^{(m+1)}\left(\alpha\right).
\end{equation}

Now we prove \eqref{eq:Nalpha} by computing the Taylor's polynomial expression of 
$N(x,y)= y q(x,y) - p(y)$  at the point $(\alpha,\alpha)$. Of course we have 
\begin{equation}\label{eq:N}
N(x,y)=\sum_{m=1}^{k} \frac{1}{m!}\sum_{\ell=0}^{m} \left(\begin{array}{c}m \\\ell \end{array}\right) \frac{\partial^m N}{\partial x^{m-\ell}\partial y^\ell}\left(\alpha,\alpha\right) \ (x-\alpha)^{m-\ell}(y-\alpha)^{\ell}.
\end{equation}
Since $N(x,y)=yq(x,y)-p(y)$ we have that
$$
\displaystyle
\left.\begin{array}{rcl}
\frac{\partial^{\ell} N}{\partial y^{\ell}}(x,y) & = & y \frac{\partial ^{\ell} q}{\partial y^{\ell}} (x,y) + \ell \frac{\partial ^{\ell-1} q}{\partial y^{\ell-1}} (x,y) - p^{(\ell)}(y),\quad \ \  \ell >0 \\
\frac{\partial^{m} N}{\partial x^{m}}(x,y) 
& = & y  \frac{\partial ^{m} q}{\partial x^{m}} (x,y),\qquad\quad\qquad\qquad\quad\quad\quad\quad \ m>0 \\
 \frac{\partial^{m} N}{\partial x^{m-\ell} \partial y^{\ell}}(x,y) 
& = & y  \frac{\partial ^{m} q}{\partial x^{m-\ell} \partial y^{\ell}} (x,y) + \ell \frac{\partial ^{m-1} q}{\partial x^{m-\ell} \partial y^{\ell-1}} (x,y), \quad  m-\ell >0, \ \ell \geq 0.
\end{array}\right.
$$
Now we want to evaluate the expressions above at the point $(x,y)=(\alpha,\alpha)$. Since by definition $q(x,y)=D(x,y)$ we might use (\ref{eq:Dalpha}) to compute the desired derivates. Let $m,\ell\in \mathbb N $ with $0\leq\ell \leq m$. 
{\footnotesize 
\begin{equation}\label{eq:partialsN}
\displaystyle
\frac{\partial^{m} N}{\partial x^{m-\ell} \partial y^{\ell}}(\alpha,\alpha) = \left \{ 
\begin{array}{ll}
0 &  \hbox{ for } m < d-1\\
\alpha  \frac{\partial ^{m} D}{\partial x^{m-\ell} \partial y^{\ell}} (\alpha,\alpha) &  \hbox{ for } m  = d -1\\
 \alpha  \frac{\partial ^{m} D}{ \partial y^{m}} (\alpha,\alpha) + m \frac{\partial ^{m-1} D}{ \partial y^{m-1}} (\alpha,\alpha) - p^{(m)}(\alpha) &   \hbox{ for }  m> d - 1, \, m - \ell = 0 \\
  \alpha  \frac{\partial ^{m} D}{\partial x^{m}} (\alpha,\alpha)  &  \hbox{ for } m > d-1,\ \ell=0 \\ 
 \alpha  \frac{\partial ^{m} D}{\partial x^{m-\ell} \partial y^{\ell}} (\alpha,\alpha) + \ell \frac{\partial ^{m-1} D}{\partial x^{m-\ell} \partial y^{\ell-1}} (\alpha,\alpha) &  \hbox{ for } m > d-1 ,\, \ell \geq 1.
\end{array}
\right.
\end{equation}}
From \eqref{eq:Dalpha}, \eqref{eq:derivD}  and \eqref{eq:partialsN} we can compute the partial derivatives of 
\eqref{eq:N} depending on $m$ and $\ell$ to get $N(x,y)=\alpha D(x,y)+ N_1(x,y)$.
\end{proof}

Next two lemmas deal with the partial derivatives of the polynomials $N(x,y)$ and $D(x,y)$ at points of the form 
$(\alpha_1,\alpha_2)$ where $\alpha_1$ and $\alpha_2$ are different real roots of $p$ of multiplicity $d_1\geq 1$ and $d_2\geq 1$, that is $p^{(j)}(\alpha_k)=0$ for $0 \leq j \leq d_k-1$ and $p^{(d_k)}(\alpha_k)\neq 0$,  $k=1,2$. Notice that $D(x,y)=q(x,y)$ and $N(x,y)=yq(x,y)-p(y)$.

\begin{lemma} \label{eq:lemma_D_focal}
Let $p$ a polynomial of degree $k$ and let $\alpha_1$ and $\alpha_2$ be two different real roots of $p$ with multiplicity $d_1$ and $d_2$, respectively. Let $m,\ell\in \mathbb N $ with $0<\ell < m$. Then 
\begin{equation}\label{eq:qalpha_1alpha_2}
\begin{split}
&\frac{\partial^m q}{\partial x^m}\ab = \frac{1}{\alpha_1-\alpha_2}\left(p^{(m)}\left(\alpha_1\right) -m\frac{\partial^{m-1}q}{\partial x^{m-1}}\ab\right), \\
&\frac{\partial^m q}{\partial y^m}\ab = -\frac{1}{\alpha_1-\alpha_2}\left(p^{(m)}\be + m\frac{\partial^{m-1}q}{\partial y^{m-1}}\ab\right),\\
&\frac{\partial^m q}{\partial x^{m-\ell}\partial y^{\ell}}\ab = \frac{1}{\alpha_1-\alpha_2}\left(\ell\frac{\partial^{m-1}q}{\partial x^{m-\ell} \partial y^{\ell-1}}\ab-(m-\ell)\frac{\partial^{m-1}q}{\partial x^{m-\ell-1} \partial y^\ell}\ab\right). 
\end{split}
\end{equation}
\end{lemma}

\begin{proof}
From Lemma \ref{lem:pol_q}(b) we know that $(x-y)q(x,y)=p(x)-p(y)$. On the one hand we can write this expression in the following form
\begin{equation}\label{eq:px-py}
(x-\alpha_1)q(x,y)-(y-\alpha_2)q(x,y)+(\alpha_1-\alpha_2)q(x,y)=p(x)-p(y), 
\end{equation}
and on the other hand we have the Taylor's polynomial of the relevant functions

\begin{equation}\label{eq:Nab}
\begin{split}
& p(x)-p(y)= \sum_{m=0}^k \frac{1}{m!}p^{(m)}\al (x-\alpha_1)^{m}-\sum_{m=0}^k \frac{1}{m!}p^{(m)}\be (y-\alpha_2)^{m}, \\
&q(x,y)=\sum_{m=1}^{k} \frac{1}{m!}\sum_{\ell=0}^{m} \left(\begin{array}{c}m \\\ell \end{array}\right) \frac{\partial^m q}{\partial x^{m-\ell}\partial y^\ell}\ab \ (x-\alpha_1)^{m-\ell}(y-\alpha_2)^{\ell}. 
\end{split}
\end{equation}
From \eqref{eq:Nab} we can solve \eqref{eq:px-py} term by term: $(x-\alpha_1)^m$, $(y-\alpha_2)^m$ and $(x-\alpha_1)^{m-\ell}(x-\alpha_2)^\ell$, with $m,\ell \in \mathbb N$ and $0<\ell <m$. For instance from \eqref{eq:Nab} the coefficient of $(x-\alpha_1)^m$ in the left hand side of \eqref{eq:px-py} is
$$
\frac{1}{(m-1)!}\left(\begin{array}{c}m-1 \\ 0 \end{array}\right) \frac{\partial^{m-1}q}{\partial x^{m-1}}\ab+\left(\alpha_1-\alpha_2\right)\frac{1}{m!}\left(\begin{array}{c}m \\ 0 \end{array}\right) \frac{\partial^{m}q}{\partial x^{m}}\ab
$$
while the coefficient of $(x-\alpha_1)^m$ in the right hand side of \eqref{eq:px-py} is
$$
\frac{1}{m!}p^{(m)}\al.
$$
This gives the first equality in \eqref{eq:qalpha_1alpha_2}. We left the other computations to the reader. 
\end{proof}

Notice that $D(x,y)=q(x,y)$, and so the previous lemma gives explicit recursive expressions of the partial derivatives of $D(x,y)$. Similarly we can prove explicit recursive expressions of the partial derivatives of $N(x,y)$

\begin{lemma} \label{eq:lemma_N_focal}
Let $p$ a polynomial of degree $k$ and let $\alpha_1$ and $\alpha_2$ be two different real roots of $p$ with multiplicity $d_1$ and $d_2$, respectively. Let $m,\ell\in \mathbb N $ with $0<\ell < m$. Then 
\begin{equation}\label{eq:DerivativesN}
\begin{split}
&\frac{\partial^m N}{\partial x^m}\ab = \alpha_2 \frac{\partial^m q}{\partial x^m}\ab, \\
& \frac{\partial^m N}{\partial y^m}\ab = m \frac{\partial^{m-1} q}{\partial y^{m-1}}\ab + \alpha_2 \frac{\partial^m q}{\partial y^m}\ab  -p^m \be, \\
&\frac{\partial^m N}{\partial x^{m-\ell}\partial y^{\ell}}\ab = \ell\frac{\partial^{m-1}q}{\partial x^{m-\ell} \partial y^{\ell-1}}\ab + \alpha_2 \frac{\partial^{m}q}{\partial x^{m-\ell} \partial y^\ell}\ab.
\end{split}
\end{equation}
\end{lemma}

\begin{proof}
The proof follows the same strategy of the previous lemma noticing that 
$$
N(x,y)=yq(x,y)-p(y)=(y-\alpha_2)q(x,y)-\alpha_2 q(x,y)-p(y)
$$
and resolving term by term. 
\end{proof}

\section{Local behaviour of the secant map near focal points and multiple roots}

Our main goal in this section is to study, using the Taylor's polynomials described in the previous section, the local behaviour of the secant map at two different type of points: $(\alpha,\alpha)$ with 
$\alpha$ being a root of $p$ of multiplicity $d >1$, and $(\alpha_1,\alpha_2)$ with $\alpha_j$ being a root of $p$ with multiplicity $d_j,\ j=1,2$. 

Let $\Gamma_{m,\kappa,\tau,\sigma}(t)=(\xi(t),\mu_{m,\kappa,\tau,\sigma}(t))$ be a curve passing through $(0,0)$ at $t=0$ with  
\begin{equation}\label{eq:curve}
\begin{split}
\xi(t)=& \ t + \frac{1}{2} t^2 +\frac{1}{6}t^3+\frac{1}{24}t^4+O\left(t^5\right)  \\
\mu_{m,\kappa,\tau,\sigma}(t)=& \ m t + \frac{\kappa}{2} t^2 +\frac{\tau}{6}t^3+\frac{\sigma}{24}t^4+O\left(t^5\right), 
\end{split}
\end{equation}
where $m$ (the slope), $\kappa$ (the curvature), $\tau$ (the torsion) and $\sigma$ are real parameters. If no confusions arise we will not show the dependence of the curve on the parameters.

To simplify the exposition  we introduce the  following auxiliary map $A_1(t)=\xi(t)\mu(t)$ and the parameter $\lambda_k=\frac{1}{k!}p^{(k)}\left(\alpha\right)$. 

\begin{lemma}\label{lem:tech}
Let $\Gamma(t) $ as in \eqref{eq:curve}. Then,
\[
S\left(\xi(t)+\alpha,\mu(t)+\alpha\right)= \left( \mu(t)+\alpha,\frac{A(t)}{B(t)}+ \alpha \right). 
\]
where 
{\footnotesize
\begin{equation}\label{eq:AB}
\begin{split}
&A(t)=t^dA_1(t)
\sum_{m=d}^{k} \lambda_{m}  t^{m-d} \sum_{\ell=1}^{m-1}\left(1+\frac{1}{2}t+\frac{1}{6}t^2+O(t^3)\right)^{m-1-\ell}\left(m+\frac{\kappa}{2}t+\frac{\tau}{6}t^2+O(t^3)\right)^{\ell-1}
\\
&B(t)=t^{d-1}\sum_{m=d}^{k} \lambda_{m}  t^{m-d} \sum_{\ell=0}^{m-1}\left(1+\frac{1}{2}t+\frac{1}{6}t^2+\frac{1}{12}t^3+O(t^4)\right)^{m-1-\ell}\left(m+\frac{\kappa}{2}t+\frac{\tau}{6}t^2+\frac{\sigma}{12}t^3+O(t^4)\right)^{\ell}.
\end{split}
\end{equation}
}
Moreover,
{\footnotesize
\begin{equation}\label{eq:lim}
\displaystyle
\frac{A(t)}{B(t)}=\frac{tA_1(t)
\sum_{m=d}^{k} \lambda_{m}  t^{m-d} \sum_{\ell=1}^{m-1}\left(1+\frac{1}{2}t+\frac{1}{6}t^2+O(t^3)\right)^{m-1-\ell}\left(m+\frac{\kappa}{2}t+\frac{\tau}{6}t^2+O(t^3)\right)^{\ell-1}}{\sum_{m=d}^{k} \lambda_{m}  t^{m-d} \sum_{\ell=0}^{m-1}\left(1+\frac{1}{2}t+\frac{1}{6}t^2+\frac{1}{12}t^3+O(t^4)\right)^{m-1-\ell}\left(m+\frac{\kappa}{2}t+\frac{\tau}{6}t^2+\frac{\sigma}{12}t^3+O(t^4)\right)^{\ell}}.
\end{equation}
}

\end{lemma}

\begin{proof}
We focus on the second component of the secant map.  From Lemma \ref{lem:N_D} we have 
{\footnotesize 
\begin{equation*}
  \frac{N\left(\xi(t)+\alpha,\mu(t)+\alpha\right)}{D\left(\xi(t)+\alpha,\mu(t)+\alpha\right)}=  \alpha+ \frac{N_1\left(\xi(t)+\alpha,\mu(t)+\alpha\right)}{D\left(\xi(t)+\alpha,\mu(t)+\alpha\right)} = 
\displaystyle \alpha+ \frac{\xi(t)\mu(t)\left[\sum_{m=d}^{k} \lambda_{m} \sum_{\ell=1}^{m-1} \xi^{m-1-\ell}(t)\mu^{\ell-1}(t)\right]}{\sum_{m=d}^{k} \lambda_{m} \sum_{\ell=0}^{m-1} \xi^{m-1-\ell}(t)\mu^{\ell}(t)}.
\end{equation*}}

Easy computations show that substituting the expressions of $\xi(t)$ and $\mu(t)$ on the right hand side of the above expression we get \eqref{eq:AB} and simplifying the factor $t^{d-1}$ in $A(t)$ and $B(t)$ we obtain \eqref{eq:lim}.
\end{proof}

\begin{lemma}\label{lem:d_odd}
Let $d\geq 3$ be an odd number and assume  $\alpha$ is a multiple root of $p$ of multiplicity $d$. Then 
\begin{equation*}
\lim_{t \to 0} S\left(\xi(t)+\alpha,\mu(t)+\alpha\right)= \left(\alpha,\alpha\right).
\end{equation*}
\end{lemma}

\begin{proof}

Using the above lemma it is enough to show that 
\[
 \lim_{t \to 0} \frac{A(t)}{B(t)} = 0.
\]
On the one hand the numerator of \eqref{eq:lim} tends to 0 as $t \to 0$. On the other hand the denominator writes as 
\begin{equation}\label{eq:Gd}
\lambda_{d}\left[1+m+\ldots+m^{d-1}\right]+O(t).
\end{equation}
We claim that if $d$ is an odd number then $\lambda_{d}\left(1+m+\ldots+m^{d-1}\right)$  is different from zero. The claim follows from the fact that $\alpha$ is a root of $p$ of multiplicity $d$ ($\lambda_d \neq0)$ and 
\begin{equation}\label{eq:dodd}
G_d(m):=1+m+\ldots + m^{d-1}=
\begin{cases}
      d & \ {\rm if} \ $m=1$ \\
      \frac{1-m^d}{1-m} & {\rm otherwise}.
    \end{cases}
\end{equation}
\end{proof}

\begin{lemma}\label{lem:d_even}
Let $d\geq 2$ be an even number and assume  $\alpha$ is a multiple root of $p$ of multiplicity $d$.   The following statements hold.
\begin{itemize}
\item[(a)] If $m\ne -1$ then
\begin{equation*}
\lim_{t \to 0} S\left(\xi(t)+\alpha,\mu(t)+\alpha\right)= \left(\alpha,\alpha\right).
\end{equation*}
\item[(b)] If $m=-1$  then 
\begin{equation*}
\lim_{t \to 0} S\left(\xi(t)+\alpha,\mu_{-1,\kappa,\tau,\sigma}(t)+\alpha\right)= \left(\alpha,y_\kappa\right),
\end{equation*}
and the map $\kappa \mapsto y_{\kappa}$ is one-to-one. Moreover,  fixing any value of  $\kappa \neq -1$ and given any pair of values  ${\bf m, s}\in \mathbb R$ there exists a unique pair $\tau_\kappa,\sigma_\kappa \in \mathbb R$ such that $S(\Gamma_{-1,\kappa,\tau_\kappa,\sigma_\kappa})$ is a curve passing through the point $\left(\alpha,y_{\kappa}\right)$ with slope ${\bf m}$ and curvature ${\bf s}$.
 \end{itemize}
\end{lemma}

\begin{proof}
The proof of statement (a), $m\ne -1$, follows similarly as in the previous lemma. The equalities and expressions  \eqref{eq:AB}, \eqref{eq:lim},  \eqref{eq:Gd} and \eqref{eq:dodd} are exactly the same. The polynomial $G_d$ for $d\geq 2$ even  has a unique real zero at $m=-1$. Hence for $m\ne -1$ the same arguments as before imply statement (a).

We turn our attention to the case when  $m=-1$. Set 
$$
C(\kappa)=\frac{d\lambda_d}{4}(\kappa+1)+\lambda_{d+1}.
$$

From  Lemma \ref{lem:tech}, some computations show that 
\begin{equation*}
\begin{split}
&A(t)=
-\lambda_dt+\frac{d\lambda_d}{4}(\kappa-1)t^2+R_1(\kappa,\tau)t^3+O(t^4),
\\
&B(t)=Ct+\left(R_2(\kappa)+\frac{d\lambda_d}{12}(1+\tau)\right)t^2+
\left(R_3(\kappa,\tau)+\frac{d\lambda_d}{24}(1+\sigma)\right)t^3+O(t^4),
\\
&\frac{A(t)}{B(t)}=-\lambda_d\frac{1}{C}+\frac{1}{C^2}\left(R_4(\kappa)+\frac{d}{4}\lambda_d^2(1+\tau)\right)t+\frac{1}{C^3}
\left(R_5(\kappa,\tau)+\frac{d}{24}\lambda_d(1+\sigma)\right)t^2+O(t^3),
\end{split}
\end{equation*}
where $R_j(\kappa,\tau), \, j = 1, \ldots, 5$ are polynomials  whose coefficients depend on $\lambda_d$, $\lambda_{d+1}$ and $\lambda_{d+2}$.
Consequently,
\begin{equation*}
\lim_{t \to 0} S\left(\xi(t)+\alpha,\mu_{-1,\kappa,\tau,\sigma}(t)+\alpha\right)= \left(\alpha,\alpha+\lim_{t \to 0} \frac{A(t)}{B(t)}\right)
=\left(\alpha,\alpha- \frac{\lambda_d}{C(\kappa)}\right).
\end{equation*}
This proves that the map $\kappa \mapsto y_\kappa:=\alpha- \lambda_d/C(\kappa)$ is one to one. Since the parameters $\tau$ and $\sigma$ appear linearly on the expression of $A(t)/B(t)$ it is easy to see that for any $\kappa \neq -1$, we might arrange the values of $\tau$ and $\sigma$ to make sure that the slope and curvature of the curve $S\left(\alpha+\xi,\alpha+ \mu_{-1,\kappa,\tau,\sigma}(t)\right)$ meet any pair $\{{\bf m}, {\bf s}\}$.
\end{proof}

\section{Proof of Theorem A}

We denote by $D\left(\left(\alpha, \alpha\right),\varepsilon\right)$ the disc centered at $\left(\alpha, \alpha\right)$ of radius $\varepsilon>0$ and by ${\rm dist}$ the Euclidian distance. The proof of Theorem A splits into two lemmas.  
\begin{lemma}
Let $p$ a polynomial of degree $k$ and let $\alpha$ be a real root of $p$ of multiplicity $d\geq 1$. Set $\mathcal Q=\{(x,y)\in \mathbb R^2 \ |  \ x \geq \alpha \ {\rm and} \ y \geq \alpha \}$. Let $\varepsilon>0$ small enough. The following statements hold.
\begin{itemize}
\item[(a)] If $d$ is an odd number then $D\left(\left(\alpha, \alpha\right),\varepsilon\right) \subset  \mathcal A(\alpha)$. 
\item[(b)] If $d$ is an even number then $D\left(\left(\alpha, \alpha\right),\varepsilon\right) \cap \mathcal Q \subset \mathcal A(\alpha)$. Moreover $\left(\alpha, \alpha\right) \in \partial  \mathcal A(\alpha)$.
\end{itemize}
\end{lemma}

\begin{proof}
If $d=1$ this follows from 
\cite[Theorem A(a)]{Tangent}.  

So we first assume $d>1$ is an odd number. From Lemma \ref{lem:d_odd} we might  extend continuously the map $S$ at the point $(\alpha,\alpha)$ by defining $S(\alpha,\alpha)=(\alpha,\alpha)$. 
We claim that for sufficiently small values of $\xi,\mu \in \mathbb R$ we have 
\begin{equation}\label{eq:dist}
{\rm dist}\left(S\left(\alpha+\xi,\alpha+\mu\right),\left(\alpha,\alpha\right)\right) \leq  {\rm dist}\left(\left(\alpha,\alpha\right),\left(\alpha+\xi,\alpha+\mu\right)\right).
\end{equation}
To see the claim we use Lemma \ref{lem:N_D} to show that
\begin{equation} \label{eq:Secant_S}
S\left(\alpha+\xi,\alpha+\mu\right)=\left(\alpha+\mu,\alpha + \frac{N_1\left(\alpha+\xi,\alpha+\mu\right)}{D\left(\alpha+\xi,\alpha+\mu\right)}\right),
\end{equation}
where
{\footnotesize
\begin{equation}
\begin{split}
&N_1\left(\alpha+\xi,\alpha+\mu\right)=\xi \mu\sum_{m=d}^{k-1} \frac{p^{(m)}\left(\alpha\right)}{m!} \sum_{\ell=1}^{m-1}\xi^{m-1-\ell}\mu^{\ell-1} =  \xi \mu\frac{p^{(d)}\left(\alpha\right)}{d!} \sum_{\ell=1}^{d-1}\xi^{d-1-\ell}\mu^{\ell-1}  + O\left(|\xi|+|\mu|\right)^{d+1},\\
&D\left(\alpha+\xi,\alpha+\mu\right)=\sum_{m=d}^{k} \frac{p^{(m)}\left(\alpha\right)}{m!}\sum_{\ell=0}^{m-1}\xi^{m-1-\ell}\mu^{\ell} =  \frac{p^{(d)}\left(\alpha\right)}{d!} \sum_{\ell=0}^{d-1}\xi^{d-1-\ell}\mu^{\ell}+ O\left(|\xi|+|\mu|\right)^{d}.
\end{split}
\end{equation}
}
On the one hand observe that 
$$
S\left(\alpha,\alpha+\mu\right)=\left(\alpha+\mu,\alpha\right) \quad {\rm and} \quad S\left(\alpha+\xi,\alpha\right)=\left(\alpha,\alpha\right),
$$
and so \eqref{eq:dist} is satisfied on those lines with equality.
On the other hand if $\xi\mu\ne 0$  
\begin{equation}\label{eq:approx}
S\left(\alpha+\xi,\alpha+\mu\right)\approx \left(\alpha+\mu,\alpha + 
\mu\xi \frac{\sum_{\ell=1}^{d-1}\xi^{d-1-\ell}\mu^{\ell-1}}{\sum_{\ell=0}^{d-1}\xi^{d-1-\ell}\mu^{\ell}}\right).
\end{equation}
Hence \eqref{eq:dist} is satisfied if and only if
\begin{equation}\label{eq:menor1}
\mu \frac{\sum_{\ell=1}^{d-1}\xi^{d-1-\ell}\mu^{\ell-1}}{\sum_{\ell=0}^{d-1}\xi^{d-1-\ell}\mu^{\ell}} < 1.
\end{equation}
Since $d$ is odd  we have from \eqref{eq:dodd} that the denominator of \eqref{eq:approx} is bounded away from zero and always positive. So a sufficient condition to satisfy the above inequality is
$$
\sum_{\ell=1}^{d-1}\xi^{d-1-\ell}\mu^{\ell} < \sum_{\ell=0}^{d-1}\xi^{d-1-\ell}\mu^{\ell},
$$
which is an immediate exercise. 

Second suppose $d>1$ is an even number. All inequalities above works as well and the denominator of \eqref{eq:approx} is bounded away from zero and it is always positive as long as $\xi$ and $\mu$ are positive numbers. So the same conclusion as before is obtained for points in $D\left(\left(\alpha, \alpha\right),\varepsilon\right) \cap \mathcal Q$. Notice, however, that Lemma \ref{lem:d_even} implies that there are curves (all with slope $m=-1$) passing through $\left(\alpha,\alpha\right)$  whose images by $S$ are curves passing through any point of the form $(\alpha,y),\ y\in \mathbb R$. Hence we conclude that $\left(\alpha, \alpha\right) \in \partial \mathcal A(\alpha)$.
\end{proof}

Statement (a) of the lemma above implies statement (a) of Theorem A. Moreover from statement (b) of the lemma above, to finish the proof of Theorem A all we need to do is to show that $\left(\alpha, \alpha\right) \in \partial \mathcal A(\beta)$, for all $\beta\ne \alpha$ a simple root of $p$.

\begin{lemma}
Let $\alpha$ a root of $p$ of even multiplicity $d_\alpha$, and let $\beta\ne \alpha$ any simple real root of $p$. Then $\left(\alpha,\alpha\right) \in \partial A\left(\beta\right)$.
\end{lemma}

\begin{proof} 
We claim that there exist curves passing through $\left(\alpha,\alpha\right)$ whose second image by $S$ correspond to curves passing through  points $(\beta,y)$ for almost every $y\in \mathbb R$. Since points in this vertical line  (except a finite number) belong to  $\mathcal A\left(\beta\right)$ the lemma follows. We see the claim into two steps. 

First, observe that Lemma \ref{lem:d_even}(b) implies that, choosing parameters for a curve $\hat{\Gamma}(t)$ passing through $\left(\alpha,\alpha\right)$, its image $S\left(\hat{\Gamma}(t)\right)$ is a curve through the point $(\alpha,\beta)$ with arbitrary slope and curvature.  

Second, let us consider an arbitrary curve $\Gamma(t)$ passing through the point $(\alpha,\beta)$ with slope $m=0$ and curvature $\kappa \in \mathbb R$. Our goal is to show that varying $\kappa\in \mathbb R$ the image curve   
$S\left(\Gamma(t)\right)$ is a curve passing through $(\beta,y_\kappa),\ y_\kappa\in \mathbb R$, as desired.

To simplify the computations consider the curve in \eqref{eq:curve} of the form $\Gamma_{0,\kappa,0,0}$ ignoring the higher order terms; that is,
\begin{equation}\label{eq:curve_0}
\xi(t)= \ t + \frac{1}{2} t^2,  \quad \mu_{0,\kappa,0,0}(t)= \frac{\kappa}{2} t^2\ . 
\end{equation}
Then 
\begin{equation*}
\lim_{t \to 0} S\left(\xi(t)+\alpha,\mu(t)+\beta\right)= \left(\beta,\lim_{t \to 0} \frac{N\left(\xi(t)+\alpha,\mu(t)+\beta\right)}{D\left(\xi(t)+\alpha,\mu(t)+\beta\right)}\right).
\end{equation*}
\end{proof}
The Taylor's polynomial of $N$ and $D$ at a point $\left(\alpha,\beta\right)$ (see Lemmas \ref{eq:lemma_D_focal} and \ref{eq:lemma_N_focal} for the expressions of the partial derivatives) we get
\begin{equation*}
\begin{split}
&  N\left( t + \frac{1}{2} t^2+\alpha,\frac{\kappa}{2} t^2+\beta\right)=\frac{1}{2\left(\alpha-\beta\right)}\left(\beta p^{\prime\prime}\left(\alpha\right)-\alpha p^{\prime}\left(\beta\right)\kappa\right)t^2+O(t^3) \\
& D\left( t + \frac{1}{2} t^2+\alpha,\frac{\kappa}{2} t^2+\beta\right)=\frac{1}{2\left(\alpha-\beta\right)}\left(p^{\prime\prime}\left(\alpha\right)-p^{\prime}\left(\beta\right)\kappa\right)t^2+O(t^3). 
\end{split}
\end{equation*}
Thus
\begin{equation*}
\lim_{t \to 0} S\left(\xi(t)+\alpha,\mu(t)+\beta\right)= \left(\beta,\frac{\beta p^{\prime\prime}\left(\alpha\right)-\alpha p^{\prime}\left(\beta\right)\kappa}{p^{\prime\prime}\left(\alpha\right)-p^{\prime}\left(\beta\right)\kappa}\right),
\end{equation*}
and since $p^{\prime}(\beta)\ne 0$ the result follows.

\end{document}